\documentclass[a4paper,12pt]{amsart}

\usepackage[english]{babel} 
\usepackage{amsmath,amsthm,amsfonts,amssymb,units}
\usepackage{xcolor,tikz,tikz-3dplot,subcaption}

\newtheorem{thm}{Theorem}

\newtheorem{cor}[thm]{Corollary}
\newtheorem{lem}[thm]{Lemma}
\theoremstyle{definition}
\newtheorem{rem}[thm]{Remark}

\newtheorem{exam}[thm]{Example}

\newcommand{\brac}[1]{\left(#1\right)}

\def\Om{\Omega}
\def\om{\omega}
\def\p{\partial}
\def\reals{\mathbb{R}}
\def\di{\,{\rm d}}
\def\phih{\hat{\phi}}

\def\cf{c_{\rm f}}
\def\cp{c_{\rm p}}
\def\cm{c_{\rm m,1}}
\def\cmm{c_{\rm m,2}}
\def\cmt{c_{\rm m,t}}
\def\cmn{c_{\rm m,n}}
\newcommand{\cpw}[1]{c_{{\rm p},#1}}
\def\ol{\overline}

\newcommand{\set}[2]{\{#1\,\mid\,#2\}}
\newcommand{\bset}[2]{\bigg\{#1\,\bigg|\,#2\bigg\}}

\def\na{\nabla}
\DeclareMathOperator{\opdiv}{div}
\def\div{\opdiv}
\DeclareMathOperator{\rot}{rot}
\DeclareMathOperator{\diam}{diam}

\DeclareMathOperator{\cont}{\sf C}
\DeclareMathOperator{\lebesgue}{\sf L}
\DeclareMathOperator{\hilbert}{\sf H}
\DeclareMathOperator{\divergence}{\sf D}
\DeclareMathOperator{\rotation}{\sf R}

\def\DF{\mathcal H_{\rm D}}
\def\NF{\mathcal H_{\rm N}}

\def\cic{\mathring{\cont}{}^\infty}

\def\lo{\lebesgue^1}					% L^1
\def\lt{\lebesgue^2}					% L^2
\def\ltz{\lt_0}

\def\ho{{\hilbert^1}}				% H^1
\def\hoc{\mathring{\hilbert}{}^1}

\def\d{\divergence}					% H(div)
\def\dz{\d_0}
\def\dc{\mathring{\d}}
\def\dcz{\dc_0}

\def\r{\rotation}					% H(rot)
\def\rz{\r_0}
\def\rc{\mathring{\r}}
\def\rcz{\rc_0}

\newcommand{\norm}[1]{|#1|}
\newcommand{\normlt}[1]{\norm{#1}_{\lt}}

\newcommand{\scp}[2]{\langle#1,#2\rangle}
\newcommand{\scplt}[2]{\scp{#1}{#2}_{\lt}}

\begin{document}

\title[\sc On Helmholtz Decompositions for Bounded Domains in $\reals^3$]
{\Large\sf A Short Note on Helmholtz Decompositions for Bounded Domains in $\reals^3$}

\author{Immanuel Anjam}
\address{Taitoniekantie 9 F as. 401\\40740 Jyv\"askyl\"a\\Finland}
\email{immanuel.anjam@gmail.com}

% 35A23  	Inequalities involving derivatives and differential and integral 
%			operators, inequalities for integrals
% 35Q61  	Maxwell equations
\date{13 September 2018}

\begin{abstract}
In this short note we consider several widely used $\lt$-orthogonal Helmholtz decompositions for bounded domains in $\reals^3$. It is well known that one part of the decompositions is a subspace of the space of functions with zero mean. We refine this global property into a local equivalent: we show that functions from these spaces have zero mean in every part of specific decompositions of the domain.

An application of the zero mean properties is presented for convex domains. We introduce a specialized Poincar\'e-type inequality, and estimate the related unknown constant from above. The upper bound is derived using the upper bound for the Poincar\'e constant proven by Payne and Weinberger. This is then used to obtain a small improvement of upper bounds of two Maxwell-type constants originally proven by Pauly.

Although the two dimensional case is not considered, all derived results can be repeated in $\reals^2$ by similar calculations.
\end{abstract}

\maketitle
\tableofcontents

\vskip+6em
\begin{verbatim}
This short note was sent to Mathematica Scandinavica on 21 December 2017, and
was accepted on 24 April 2018. Estimated publication date is in 2020.
\end{verbatim}

\newpage

%====================================================================
%====================================================================

\section{Notation and Helmholtz Decompositions}

Let $\om\subset\reals^3$ be a bounded open set. The space of scalar- or vector-valued smooth functions with compact supports in $\om$ is denoted by $\cic(\om)$. We denote by $\norm{\,\cdot\,}_{\lo(\om)}$ the norm for functions in $\lo(\om)$, and by $\scp{\,\cdot\,}{\,\cdot\,}_{\lt(\om)}$ and $\norm{\,\cdot\,}_{\lt(\om)}$ the inner product and norm for functions in $\lt(\om)$. The space of scalar-valued functions in $\lt(\om)$ with zero mean is defined as
\begin{equation*}
	\ltz(\om) := \bset{\varphi\in\lt(\om)}{ \int_\om \varphi \di x = 0 } ,
\end{equation*}
and as usual, for a vector-valued function $\phi$ we write $\phi\in\ltz(\om)$ if all its components belong to $\ltz(\om)$.

Throughout this note $\Om$ denotes a bounded domain in $\reals^3$, and from now on, whenever $\om=\Om$, we sometimes omit the indication of the set in our notation.

Aside from the gradient $\na$ we will also need the divergence operator $\div$ and the rotation operator $\rot$ acting on vector-valued functions. For smooth functions they are defined as
\begin{equation*}
	\div \begin{pmatrix} \phi_1 \\ \phi_2 \\ \phi_3 \end{pmatrix}
	:= \p_1 \phi_1 + \p_2 \phi_2 + \p_3 \phi_3 ,
	\qquad
	\rot \begin{pmatrix} \phi_1 \\ \phi_2 \\ \phi_3 \end{pmatrix}
	:= \begin{pmatrix}
	\p_2 \phi_3 - \p_3 \phi_2 \\
	\p_3 \phi_1 - \p_1 \phi_3 \\
	\p_1 \phi_2 - \p_2 \phi_1
	\end{pmatrix} .
\end{equation*}
We define the usual Sobolev spaces
\begin{align*}
	\ho	& := \set{\varphi\in\lt}{\na\varphi\in\lt} ,
		& \hoc & := \ol{\cic}^{\ho} , \\
	\d	& := \set{\phi\in\lt}{\div\phi\in\lt} ,
		& \dc & := \ol{\cic}^{\d} , \\
	\r	& := \set{\phi\in\lt}{\rot\phi\in\lt} ,
		& \rc & := \ol{\cic}^{\r} ,
\end{align*}
which are Hilbert spaces. Note that on the former spaces the differential operators are now defined in the usual weak sense. The latter spaces, where the closures are taken with respect to graph norms, generalize the classical homogenous scalar, normal, and tangential boundary conditions, respectively. The operators satisfy
\begin{align*}
	& \forall \varphi\in\hoc & & \forall \phi\in\d
	& & & \scplt{\na\varphi}{\phi} & = - \scplt{\varphi}{\div\phi} , \\
	& \forall \varphi\in\ho & & \forall \phi\in\dc
	& & & \scplt{\na\varphi}{\phi} & = - \scplt{\varphi}{\div\phi} , \\
	& \forall \phi\in\rc & & \forall \psi\in\r
	& & & \scplt{\rot\phi}{\psi} & = \scplt{\phi}{\rot\psi} .
\end{align*}
Note, that even though it is not indicated in the notation, we have two of each differential operator, one acting on a space without a boundary condition, and one acting on a space with a boundary condition. We also define
\begin{align*}
	\dz		& := \set{\phi\in\d}{\div\phi=0} ,
	& \dcz 	& := \set{\phi\in\dc}{\div\phi=0} , \\
	\rz		& := \set{\phi\in\r}{\rot\phi=0} ,
	& \rcz 	& := \set{\phi\in\rc}{\rot\phi=0} .
\end{align*}
By the projection theorem we obtain the $\lt$-orthogonal Helmholtz decompositions
\begin{align}
	\lt & = \ol{\na\hoc} \oplus \dz
			= \rcz \oplus \ol{\rot\r}
			= \ol{\na\hoc} \oplus \DF \oplus \ol{\rot\r} ,
		& \DF & := \dz \cap \; \rcz, \label{eq:hd1} \\
	\lt & = \ol{\na\ho} \oplus \dcz
			= \rz \oplus \, \ol{\rot\rc}
			= \ol{\na\ho} \oplus \NF \oplus \ol{\rot\rc} ,
		& \NF & := \dcz \cap \rz , \label{eq:hd2}
\end{align}
where $\DF$ and $\NF$ are the spaces of Dirichlet and Neumann fields, respectively. In particular, we have the decompositions
\begin{equation} \label{eq:hd3}
	\dcz = \NF \oplus \ol{\rot\rc} ,
	\qquad
	\rcz = \ol{\na\hoc} \oplus \DF .
\end{equation}

It is easy to see that functions from $\dcz$ have zero mean globally, i.e., they belong to $\ltz$:
\begin{equation} \label{eq:Dglobal}
	\forall \phi\in\dcz \qquad
	\int_\Om \phi_i \di x = \scplt{\phi}{\na x_i} = -\scplt{\div\phi}{x_i} = 0 .
\end{equation}
Similarly we see that $\rcz \subset \ltz$: for $v_1(x):=(0,0,x_2)$, $v_2(x):=(x_3,0,0)$, and $v_3(x):=(0,x_1,0)$ we have
\begin{equation*}
	\forall \phi\in\rcz \qquad
	\int_\Om \phi_i \di x = \scplt{\phi}{\rot v_i} = \scplt{\rot\phi}{v_i} = 0 .
\end{equation*}
In this note we show that functions from the above two spaces satisfy local zero mean properties with respect to certain decompositions of $\Om$.

For our considerations, it is not needed to assume any regularity of the domain. However, we mention that if $\Om$ is Lipschitz, then Rellich's selection theorem and Weck's selection theorem \cite{weck} hold. This means that the closure bars in \eqref{eq:hd1}--\eqref{eq:hd3} can be skipped, and both $\DF$ and $\NF$ are finite dimensional. Furthermore, if the domain is topologically equivalent to a ball, then $\DF=\NF=\{0\}$. For more information on Helmholtz decompositions we refer to \cite{leisbook} and \cite{paulyfirstorder}, which contains a concise exposition of Helmholtz decompositions in a general Hilbert space setting.

This note is organized as follows. Section \ref{sec:DD} contains additional notation related to decompositions of the domain. Our main results, Theorems \ref{thm:D} and \ref{thm:R}, and the local zero mean properties of Corollaries \ref{cor:D} and \ref{cor:R}, are in Section \ref{sec:ZM}. In Section \ref{sec:C} we use these results to derive, in the case of convex domains, slightly improved upper bounds of certain Maxwell-type constants related to the theory of electromagnetism.

%====================================================================
%====================================================================

\section{Decompositions of the Domain}
\label{sec:DD}

Our calculations are invariant with respect to translations of the domain, so without loss of generality we assume $\Om$ to be contained in the rectangular cuboid
\begin{equation*}
	I:=(0,l_1)\times(0,l_2)\times(0,l_3), \qquad 0<l_1,l_2,l_3<\infty .
\end{equation*}
We assume $\Om$ is translated such that $I$ is as small as possible. Note that the calculations of the following section no longer hold if the domain is rotated.

In what follows we will often need two or three distinct indices from the index set $\{1,2,3\}$. To this end, we define $\{1,2,3\}_p$ to denote the set of all $p$-permutations of the set $\{1,2,3\}$, where $p$ is either $2$ or $3$.

For $0 \leq \alpha_i < \beta_i \leq l_i$, $i\in\{1,2,3\}$, we define
\begin{align*}
	I_i & := \set{x \in I}{\alpha_i < x_i < \beta_i} ,
	& I_{ij} & := I_i \cap I_j , \\
	\qquad
	\Om_i & := \set{x \in \Om}{\alpha_i < x_i < \beta_i} ,
	& \Om_{ij} & := \Om_i \cap \Om_j ,
\end{align*}
where in the latter definitions $(i,j)\in\{1,2,3\}_2$. Note that $\Om_i \subset I_i$ and $\Om_{ij} \subset I_{ij}$ hold. Examples of these subsets are illustrated in Figure \ref{fig:cuts}. It is clear that $\Om$ can be decomposed in such pieces in a way that the pieces are nonintersecting, and that the union of their closures equals $\ol{\Om}$. Note also that if $\Om_i$ and $\Om_{ij}$ appear in the same relation, they are always related to each other, i.e., in particular $\Om_{ij}\subset\Om_i$ holds.

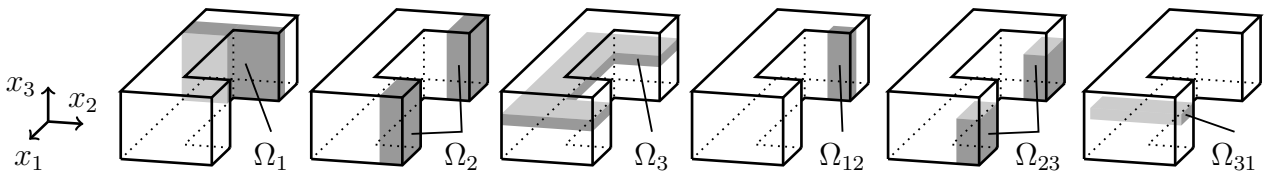
\begin{figure} [!b]
\centering
\tdplotsetmaincoords{75}{105}
%####################################################################
% axis
%####################################################################
\begin{tikzpicture}[scale=0.32,tdplot_main_coords]
	\draw [black,line width=1pt,->] (0,-3,2) -- (3,-3,2) node[anchor=north]{$x_1$};
	\draw [black,line width=1pt,->] (0,-3,2) -- (0,-1.5,2) node[anchor=south]{$x_2$};
	\draw [black,line width=1pt,->] (0,-3,2) -- (0,-3,3.5) node[anchor=east]{$x_3$};
\end{tikzpicture}
%####################################################################
% x_1-cut
%####################################################################
\begin{tikzpicture}[scale=0.31,tdplot_main_coords]
	\def\lone{7}
	\def\lthree{3}

	%x_1-cut
	\fill [black,opacity=.4] (-3,0,\lthree) -- (-3,2,\lthree) -- (-5,2,\lthree) -- (-5,0,\lthree) ;
	\fill [black,opacity=.4] (-4,2,\lthree) -- (-4,4,\lthree) -- (-5,4,\lthree) -- (-5,2,\lthree) ;
	\fill [black,opacity=.4] (-3,2,\lthree) -- (-4,2,\lthree) -- (-4,2,0) -- (-3,2,0);
	\fill [black,opacity=.4] (-4,4,\lthree) -- (-5,4,\lthree) -- (-5,4,0) -- (-4,4,0);
	\fill [black,opacity=.2] (-3,2,\lthree) -- (-3,2,0) -- (-3,0,0) -- (-3,0,\lthree);
	\fill [black,opacity=.4] (-4,4,\lthree) -- (-4,4,0) -- (-4,2,0) -- (-4,2,\lthree);

	% upper level
	\draw [black,line width=1pt] (-\lone,0,\lthree) -- (\lone,0,\lthree) -- (\lone,4,\lthree) -- (\lone-3,4,\lthree) -- (\lone-3,2,\lthree) -- (-\lone+3,2,\lthree) -- (-\lone+3,4,\lthree) -- (-\lone,4,\lthree) -- cycle;
	
	% visible vertical lines
	\draw [black,line width=1pt] (\lone,0,\lthree) -- (\lone,0,0);
	\draw [black,line width=1pt] (\lone,4,\lthree) -- (\lone,4,0);
	\draw [black,line width=1pt] (\lone-3,4,\lthree) -- (\lone-3,4,0);
	\draw [black,line width=1pt] (-\lone+3,4,\lthree) -- (-\lone+3,4,0);
	\draw [black,line width=1pt] (-\lone,4,\lthree) -- (-\lone,4,0);
	
	% visible lower level
	\draw [black,line width=1pt] (\lone,0,0) -- (\lone,4,0) -- (\lone-3,4,0);
	\draw [black,line width=1pt] (-\lone,4,0) -- (-\lone+3,4,0) -- (-\lone+3,2,0) -- (-\lone+3.5,2,0);
	
	% one invis line
	\draw [black,line width=0.7pt,dotted] (-\lone+3,2,\lthree) -- (-\lone+3,2,0);
	% bottom invis lines
	\draw [black,line width=0.7pt,dotted] (-\lone,4,0) -- (-\lone,0,0) -- (\lone,0,0);
	\draw [black,line width=0.7pt,dotted] (\lone-3,4,0) -- (\lone-3,2,0) -- (-3,2,0);
	
	% text
	\draw [black,line width=0.7pt] (-6,2,1) -- (1,5,0) node[anchor=north]{$\Om_1$};
\end{tikzpicture}
%####################################################################
% x_2-cut
%####################################################################
\begin{tikzpicture}[scale=0.31,tdplot_main_coords]
	\def\lone{7}
	\def\lthree{3}

	%x_2-cut: part 1
	\fill [black,opacity=.4] (-4,3,\lthree) -- (-4,4,\lthree) -- (-7,4,\lthree) -- (-7,3,\lthree);
	\fill [black,opacity=.4] (-4,4,\lthree) -- (-7,4,\lthree) -- (-7,4,0) -- (-4,4,0);
	\fill [black,opacity=.4] (-4,4,\lthree) -- (-4,4,0) -- (-4,3,0) -- (-4,3,\lthree);
	
	%x_2-cut: part 2
	\fill [black,opacity=.4] (4,3,\lthree) -- (4,4,\lthree) -- (7,4,\lthree) -- (7,3,\lthree);
	\fill [black,opacity=.4] (4,4,\lthree) -- (7,4,\lthree) -- (7,4,0) -- (4,4,0);
	\fill [black,opacity=.4] (7,4,\lthree) -- (7,4,0) -- (7,3,0) -- (7,3,\lthree);

	% upper level
	\draw [black,line width=1pt] (-\lone,0,\lthree) -- (\lone,0,\lthree) -- (\lone,4,\lthree) -- (\lone-3,4,\lthree) -- (\lone-3,2,\lthree) -- (-\lone+3,2,\lthree) -- (-\lone+3,4,\lthree) -- (-\lone,4,\lthree) -- cycle;
	
	% visible vertical lines
	\draw [black,line width=1pt] (\lone,0,\lthree) -- (\lone,0,0);
	\draw [black,line width=1pt] (\lone,4,\lthree) -- (\lone,4,0);
	\draw [black,line width=1pt] (\lone-3,4,\lthree) -- (\lone-3,4,0);
	\draw [black,line width=1pt] (-\lone+3,4,\lthree) -- (-\lone+3,4,0);
	\draw [black,line width=1pt] (-\lone,4,\lthree) -- (-\lone,4,0);
	
	% visible lower level
	\draw [black,line width=1pt] (\lone,0,0) -- (\lone,4,0) -- (\lone-3,4,0);
	\draw [black,line width=1pt] (-\lone,4,0) -- (-\lone+3,4,0) -- (-\lone+3,2,0) -- (-\lone+3.5,2,0);
	
	% one invis line
	\draw [black,line width=0.7pt,dotted] (-\lone+3,2,\lthree) -- (-\lone+3,2,0);
	% bottom invis lines
	\draw [black,line width=0.7pt,dotted] (-\lone,4,0) -- (-\lone,0,0) -- (\lone,0,0);
	\draw [black,line width=0.7pt,dotted] (\lone-3,4,0) -- (\lone-3,2,0) -- (-3,2,0);
	
	% text
	\draw [black,line width=0.7pt] (-6,3,1) -- (1,5,0) node[anchor=north]{$\Om_2$} -- (5.5,4,1);
\end{tikzpicture}
%####################################################################
% x_3-cut
%####################################################################
\begin{tikzpicture}[scale=0.31,tdplot_main_coords]
	\def\lone{7}
	\def\lthree{3}
	
	%x_3-cut
	\fill [black,opacity=.2] (-\lone,0,2) -- (\lone,0,2) -- (\lone,4,2) -- (\lone-3,4,2) -- (\lone-3,2,2) -- (-\lone+3,2,2) -- (-\lone+3,4,2) -- (-\lone,4,2);
	\fill [black,opacity=.4] (7,0,2) -- (7,4,2) -- (7,4,1.5) -- (7,0,1.5);
	\fill [black,opacity=.4] (-4,2,2) -- (-4,4,2) -- (-4,4,1.5) -- (-4,2,1.5);
	\fill [black,opacity=.4] (4,4,2) -- (7,4,2) -- (7,4,1.5) -- (4,4,1.5);
	\fill [black,opacity=.4] (-4,4,2) -- (-7,4,2) -- (-7,4,1.5) -- (-4,4,1.5);
	\fill [black,opacity=.4] (4,2,2) -- (-4,2,2) -- (-4,2,1.5) -- (2.2,2,1.5);

	% upper level
	\draw [black,line width=1pt] (-\lone,0,\lthree) -- (\lone,0,\lthree) -- (\lone,4,\lthree) -- (\lone-3,4,\lthree) -- (\lone-3,2,\lthree) -- (-\lone+3,2,\lthree) -- (-\lone+3,4,\lthree) -- (-\lone,4,\lthree) -- cycle;
	
	% visible vertical lines
	\draw [black,line width=1pt] (\lone,0,\lthree) -- (\lone,0,0);
	\draw [black,line width=1pt] (\lone,4,\lthree) -- (\lone,4,0);
	\draw [black,line width=1pt] (\lone-3,4,\lthree) -- (\lone-3,4,0);
	\draw [black,line width=1pt] (-\lone+3,4,\lthree) -- (-\lone+3,4,0);
	\draw [black,line width=1pt] (-\lone,4,\lthree) -- (-\lone,4,0);
	
	% visible lower level
	\draw [black,line width=1pt] (\lone,0,0) -- (\lone,4,0) -- (\lone-3,4,0);
	\draw [black,line width=1pt] (-\lone,4,0) -- (-\lone+3,4,0) -- (-\lone+3,2,0) -- (-\lone+3.5,2,0);
	
	% invis vertical line
	\draw [black,line width=0.7pt,dotted] (-\lone+3,2,\lthree) -- (-\lone+3,2,0);
	% bottom invis lines
	\draw [black,line width=0.7pt,dotted] (-\lone,4,0) -- (-\lone,0,0) -- (\lone,0,0);
	\draw [black,line width=0.7pt,dotted] (\lone-3,4,0) -- (\lone-3,2,0) -- (-3,2,0);
	
	% text
	\draw [black,line width=0.7pt] (-6,2.5,1.25) -- (1,5,0) node[anchor=north]{$\Om_3$};
\end{tikzpicture}
%####################################################################
% x_1-cut
%####################################################################
\begin{tikzpicture}[scale=0.31,tdplot_main_coords]
	\def\lone{7}
	\def\lthree{3}

	%x_1-cut
	\fill [black,opacity=.4] (-4,3,\lthree) -- (-4,4,\lthree) -- (-5,4,\lthree) -- (-5,3,\lthree) ;
	\fill [black,opacity=.4] (-4,4,\lthree) -- (-5,4,\lthree) -- (-5,4,0) -- (-4,4,0);
	\fill [black,opacity=.4] (-4,4,\lthree) -- (-4,4,0) -- (-4,3,0) -- (-4,3,\lthree);

	% upper level
	\draw [black,line width=1pt] (-\lone,0,\lthree) -- (\lone,0,\lthree) -- (\lone,4,\lthree) -- (\lone-3,4,\lthree) -- (\lone-3,2,\lthree) -- (-\lone+3,2,\lthree) -- (-\lone+3,4,\lthree) -- (-\lone,4,\lthree) -- cycle;
	
	% visible vertical lines
	\draw [black,line width=1pt] (\lone,0,\lthree) -- (\lone,0,0);
	\draw [black,line width=1pt] (\lone,4,\lthree) -- (\lone,4,0);
	\draw [black,line width=1pt] (\lone-3,4,\lthree) -- (\lone-3,4,0);
	\draw [black,line width=1pt] (-\lone+3,4,\lthree) -- (-\lone+3,4,0);
	\draw [black,line width=1pt] (-\lone,4,\lthree) -- (-\lone,4,0);
	
	% visible lower level
	\draw [black,line width=1pt] (\lone,0,0) -- (\lone,4,0) -- (\lone-3,4,0);
	\draw [black,line width=1pt] (-\lone,4,0) -- (-\lone+3,4,0) -- (-\lone+3,2,0) -- (-\lone+3.5,2,0);
	
	% one invis line
	\draw [black,line width=0.7pt,dotted] (-\lone+3,2,\lthree) -- (-\lone+3,2,0);
	% bottom invis lines
	\draw [black,line width=0.7pt,dotted] (-\lone,4,0) -- (-\lone,0,0) -- (\lone,0,0);
	\draw [black,line width=0.7pt,dotted] (\lone-3,4,0) -- (\lone-3,2,0) -- (-3,2,0);
	
	% text
	\draw [black,line width=0.7pt] (-6,3,1) -- (1,5,0) node[anchor=north]{$\Om_{12}$};
\end{tikzpicture}
%####################################################################
% x_2-cut
%####################################################################
\begin{tikzpicture}[scale=0.31,tdplot_main_coords]
	\def\lone{7}
	\def\lthree{3}

	%x_2-cut: part 1
	\fill [black,opacity=.2] (-4,3,2) -- (-4,4,2) -- (-7,4,2) -- (-7,3,2);
	\fill [black,opacity=.4] (-4,4,2) -- (-7,4,2) -- (-7,4,0) -- (-4,4,0);
	\fill [black,opacity=.4] (-4,4,2) -- (-4,4,0) -- (-4,3,0) -- (-4,3,2);
	
	%x_2-cut: part 2
	\fill [black,opacity=.2] (4,3,2) -- (4,4,2) -- (7,4,2) -- (7,3,2);
	\fill [black,opacity=.4] (4,4,2) -- (7,4,2) -- (7,4,0) -- (4,4,0);
	\fill [black,opacity=.4] (7,4,2) -- (7,4,0) -- (7,3,0) -- (7,3,2);

	% upper level
	\draw [black,line width=1pt] (-\lone,0,\lthree) -- (\lone,0,\lthree) -- (\lone,4,\lthree) -- (\lone-3,4,\lthree) -- (\lone-3,2,\lthree) -- (-\lone+3,2,\lthree) -- (-\lone+3,4,\lthree) -- (-\lone,4,\lthree) -- cycle;
	
	% visible vertical lines
	\draw [black,line width=1pt] (\lone,0,\lthree) -- (\lone,0,0);
	\draw [black,line width=1pt] (\lone,4,\lthree) -- (\lone,4,0);
	\draw [black,line width=1pt] (\lone-3,4,\lthree) -- (\lone-3,4,0);
	\draw [black,line width=1pt] (-\lone+3,4,\lthree) -- (-\lone+3,4,0);
	\draw [black,line width=1pt] (-\lone,4,\lthree) -- (-\lone,4,0);
	
	% visible lower level
	\draw [black,line width=1pt] (\lone,0,0) -- (\lone,4,0) -- (\lone-3,4,0);
	\draw [black,line width=1pt] (-\lone,4,0) -- (-\lone+3,4,0) -- (-\lone+3,2,0) -- (-\lone+3.5,2,0);
	
	% one invis line
	\draw [black,line width=0.7pt,dotted] (-\lone+3,2,\lthree) -- (-\lone+3,2,0);
	% bottom invis lines
	\draw [black,line width=0.7pt,dotted] (-\lone,4,0) -- (-\lone,0,0) -- (\lone,0,0);
	\draw [black,line width=0.7pt,dotted] (\lone-3,4,0) -- (\lone-3,2,0) -- (-3,2,0);
	
	% text
	\draw [black,line width=0.7pt] (-6,3,1) -- (1,5,0) node[anchor=north]{$\Om_{23}$} -- (5.5,4,1);
\end{tikzpicture}
%####################################################################
% x_3-cut
%####################################################################
\begin{tikzpicture}[scale=0.31,tdplot_main_coords]
	\def\lone{7}
	\def\lthree{3}
	
	%x_3-cut
	\fill [black,opacity=.2] (\lone-3,0,2) -- (\lone-1,0,2) -- (\lone-1,4,2) -- (\lone-3,4,2);
	\fill [black,opacity=.2] (6,0,2) -- (6,4,2) -- (6,4,1.5) -- (6,0,1.5);
	\fill [black,opacity=.4] (4,4,2) -- (6,4,2) -- (6,4,1.5) -- (4,4,1.5);

	% upper level
	\draw [black,line width=1pt] (-\lone,0,\lthree) -- (\lone,0,\lthree) -- (\lone,4,\lthree) -- (\lone-3,4,\lthree) -- (\lone-3,2,\lthree) -- (-\lone+3,2,\lthree) -- (-\lone+3,4,\lthree) -- (-\lone,4,\lthree) -- cycle;
	
	% visible vertical lines
	\draw [black,line width=1pt] (\lone,0,\lthree) -- (\lone,0,0);
	\draw [black,line width=1pt] (\lone,4,\lthree) -- (\lone,4,0);
	\draw [black,line width=1pt] (\lone-3,4,\lthree) -- (\lone-3,4,0);
	\draw [black,line width=1pt] (-\lone+3,4,\lthree) -- (-\lone+3,4,0);
	\draw [black,line width=1pt] (-\lone,4,\lthree) -- (-\lone,4,0);
	
	% visible lower level
	\draw [black,line width=1pt] (\lone,0,0) -- (\lone,4,0) -- (\lone-3,4,0);
	\draw [black,line width=1pt] (-\lone,4,0) -- (-\lone+3,4,0) -- (-\lone+3,2,0) -- (-\lone+3.5,2,0);
	
	% invis vertical line
	\draw [black,line width=0.7pt,dotted] (-\lone+3,2,\lthree) -- (-\lone+3,2,0);
	% bottom invis lines
	\draw [black,line width=0.7pt,dotted] (-\lone,4,0) -- (-\lone,0,0) -- (\lone,0,0);
	\draw [black,line width=0.7pt,dotted] (\lone-3,4,0) -- (\lone-3,2,0) -- (-3,2,0);
	
	% text
	\draw [black,line width=0.7pt] (5.3,4,1.8) -- (1,5,0) node[anchor=north]{$\Om_{31}$};
\end{tikzpicture}
\caption{Examples of $\Om_i$ and $\Om_{ij}$. For illustrative purposes the $\Om_{ij}$ are chosen such that they belong to the $\Om_{i}$.}
\label{fig:cuts}
\end{figure}

%====================================================================
%====================================================================

\section{Local Zero Mean Properties}
\label{sec:ZM}

In order to prove the local zero mean properties, we show that the mean value of functions from $\dc$ and $\rc$ can be locally estimated from below and above by $\lo$-norms of their divergence and rotation, respectively.

\begin{thm} \label{thm:D}
For any $\phi\in\dc(\Om)$ the estimate
\begin{equation*}
	\forall i\in\{1,2,3\}
	\qquad
	\bigg| \int_{\Om_i} \phi_i \di x \bigg|
	\leq
	(\beta_i-\alpha_i) \norm{\div\phi}_{\lo(\Om)}
\end{equation*}
holds for an arbitrary $\Om_i$.
\end{thm}

\begin{proof}
For any $\phi\in\cic(\Om)$ its zero extension $\phih:I\rightarrow\reals^3$ belongs to $\cic(I)$. By the Fundamental Theorem of Calculus the components of this extension can be represented as
\begin{align*}
	\phih_1(x_1,x_2,x_3) & = \int_0^{x_1} \p_a\phih_1(a,x_2,x_3) \di a , \\
	\phih_2(x_1,x_2,x_3) & = \int_0^{x_2} \p_b\phih_2(x_1,b,x_3) \di b , \\
	\phih_3(x_1,x_2,x_3) & = \int_0^{x_3} \p_c\phih_3(x_1,x_2,c) \di c .
\end{align*}
Using the above representations we write
\begin{align} \label{eq:D1}
	  & \pm \int_0^{x_3}\int_0^{x_2} \phih_1(x_1,b,c) \di(bc)
		\pm \int_0^{x_3}\int_0^{x_1} \phih_2(a,x_2,c) \di(ac)
	    \pm \int_0^{x_2}\int_0^{x_1} \phih_3(a,b,x_3) \di(ab) \\
	= \; & \pm \int_0^{x_3}\int_0^{x_2}\int_0^{x_1}
		\p_a\phih_1(a,b,c)
		+ \p_b\phih_2(a,b,c)
		+ \p_c\phih_3(a,b,c)
		\di(abc) \nonumber \\
	\leq \; & \norm{\div\phih}_{\lo(I)} . \nonumber
\end{align}
By choosing $x_2=l_2$ and $x_3=l_3$, the two last terms on the l.h.s. vanish, and we obtain
\begin{equation*}
	\pm \int_0^{l_3}\int_0^{l_2} \phih_1(x_1,b,c) \di(bc)
	\leq \norm{\div\phih}_{\lo(I)} .
\end{equation*}
By integrating w.r.t. $x_1$ over $(\alpha_1,\beta_1)$ we obtain
\begin{equation*}
	\pm \int_{I_1} \phih_1 \di x \leq (\beta_1-\alpha_1) \norm{\div\phih}_{\lo(I)}
	\quad \Rightarrow \quad
	\pm \int_{\Om_1} \phi_1 \di x \leq (\beta_1-\alpha_1) \norm{\div\phi}_{\lo(\Om)} ,
\end{equation*}
since the integrals are nonzero only in $\Om$. By density the latter inequality above holds for any $\phi \in \dc(\Om)$, and we have proven the assertion for $i=1$. To prove the cases $i=2$ and $i=3$, one chooses $x_1=l_1,x_3=l_3$ and $x_1=l_1,x_2=l_2$ in \eqref{eq:D1}, respectively, and proceeds in a similar manner.
\end{proof}

\begin{thm} \label{thm:R}
For any $\phi\in\rc(\Om)$ the estimate
\begin{align*}
	\forall (i,j,k)\in\{1,2,3\}_3
	\qquad
	\bigg| \int_{\Om_{jk}} \phi_i \di x \bigg|
	\leq
	(\beta_j-\alpha_j) \norm{(\rot\phi)_k}_{\lo(\Om_k)}
\end{align*}
holds for an arbitrary $\Om_{jk}$.
\end{thm}

\begin{proof}
For any $\phi\in\cic(\Om)$ its zero extension $\phih:I\rightarrow\reals^3$ belongs to $\cic(I)$. By the Fundamental Theorem of Calculus the components of this extension can be represented as
\begin{equation*}
	\phih_2(x_1,x_2,x_3) = \int_0^{x_1} \p_a\phih_2(a,x_2,x_3) \di a ,
	\qquad
	\phih_1(x_1,x_2,x_3) = \int_0^{x_2} \p_b\phih_1(x_1,b,x_3) \di b .
\end{equation*}
Using the above representations we write
\begin{align}
	& \pm \int_0^{x_2} \phih_2(x_1,b,x_3) \di b
	\mp \int_0^{x_1} \phih_1(a,x_2,x_3) \di a \label{eq:R1} \\
	= &  \pm \int_0^{x_2}\int_0^{x_1} \p_a\phih_2(a,b,x_3)
		- \p_b\phih_1(a,b,x_3) \di(ab) \nonumber \\
	\leq & \int_0^{l_2}\int_0^{l_1} |\p_a\phih_2(a,b,x_3)
		- \p_b\phih_1(a,b,x_3)| \di(ab) . \nonumber 
\end{align}
By choosing $x_1=l_1$ in \eqref{eq:R1} and integrating w.r.t. $x_3$ over $(\alpha_3,\beta_3)$, we obtain
\begin{equation*}
	\pm \int_{\alpha_3}^{\beta_3}\int_0^{l_1} \phih_1(a,x_2,x_3) \di(ax_3)
	\leq \norm{(\rot\phih)_3}_{\lo(I_3)} .
\end{equation*}
By integrating w.r.t. $x_2$ over $(\alpha_2,\beta_2)$ we obtain
\begin{equation} \label{eq:R2}
	\pm \int_{I_{23}} \phih_1 \di x
	\leq (\beta_2-\alpha_2) \norm{(\rot\phih)_3}_{\lo(I_3)}
	\quad \Rightarrow \quad
	\pm \int_{\Om_{23}} \phi_1 \di x
	\leq (\beta_2-\alpha_2) \norm{(\rot\phi)_3}_{\lo(\Om_3)} ,	
\end{equation}
since the integrals are nonzero only in $\Om$. On the other hand, by choosing $x_2=l_2$ in \eqref{eq:R1} and integrating w.r.t. $x_3$ over $(\alpha_3,\beta_3)$, we obtain
\begin{equation*}
	\pm \int_{\alpha_3}^{\beta_3}\int_0^{l_2} \phih_2(x_1,b,x_3) \di(bx_3)
	\leq \norm{(\rot\phih)_3}_{\lo(I_3)} .
\end{equation*}
By integrating w.r.t. $x_1$ over $(\alpha_1,\beta_1)$ we obtain
\begin{equation} \label{eq:R3}
	\pm \int_{I_{13}} \phih_2 \di x
	\leq (\beta_1-\alpha_1) \norm{(\rot\phih)_3}_{\lo(I_3)}
	\quad \Rightarrow \quad
	\pm \int_{\Om_{13}} \phi_2 \di x
	\leq (\beta_1-\alpha_1) \norm{(\rot\phi)_3}_{\lo(\Om_3)} ,
\end{equation}
since the integrals are nonzero only in $\Om$. By density the latter inequalities of \eqref{eq:R2} and \eqref{eq:R3} hold for any $\phi\in\rc(\Om)$, and we have proven two of the six estimates of the assertion. The remaining estimates are proven in a similar manner by repeating the proof using the representations
\begin{equation*}
	\phih_1(x_1,x_2,x_3) = \int_0^{x_3} \p_c\phih_1(x_1,x_2,c) \di c ,
	\qquad
	\phih_3(x_1,x_2,x_3) = \int_0^{x_1} \p_a\phih_3(a,x_2,x_3) \di a ,
\end{equation*}
and
\begin{equation*}
	\phih_3(x_1,x_2,x_3) = \int_0^{x_2} \p_b\phih_3(x_1,b,x_3) \di b ,
	\qquad
	\phih_2(x_1,x_2,x_3) = \int_0^{x_3} \p_c\phih_2(x_1,x_2,c) \di c .
\end{equation*}
\end{proof}

The following two corollaries are directly implied by Theorems \ref{thm:D} and \ref{thm:R}.

\begin{cor} \label{cor:D}
Let $\phi\in\dcz(\Om)$. Then $\phi_i \in \ltz(\Om_i)$ for any $\Om_i$, where $i\in\{1,2,3\}$.
\end{cor}

\begin{cor} \label{cor:R}
Let $\phi\in\rcz(\Om)$. Then $\phi_i \in \ltz(\Om_{jk})$ for any $\Om_{jk}$, where $(i,j,k)\in\{1,2,3\}_3$.
\end{cor}

\begin{rem} \label{rem:R}
Theorem \ref{thm:R} allows for more general statements about $\rc$ than Corollary \ref{cor:R}:
\begin{itemize}
\item[\bf(i)] It is easy to see that Corollary \ref{cor:R} holds not only for $\rcz$ but even for
\begin{equation*}
	\set{\psi\in\rc}{(\rot\psi)_i=(\rot\psi)_j=0, \; (i,j)\in\{1,2,3\}_2} .
\end{equation*}
\item[\bf(ii)] Even if only one component of the rotation of $\phi\in\rc$ vanishes on a subset of $\Om$, in certain cases we might still be able to obtain information about where $\phi$ has zero mean. If, for example, $(\rot\phi)_3=0$ in $\om\subset\Om$ which is a $\Om_3$-set, then Theorem \ref{thm:R} implies that $\phi_1,\phi_2\in\ltz(\om)$.
\end{itemize}
\end{rem}

%====================================================================
%====================================================================

\section{An Application for Convex Domains}
\label{sec:C}

In this section we assume the domain $\Om$ to be convex. Then $\Om$ is Lipschitz \cite{grisvard1985}, and Rellich's selection theorem and Weck's selection theorem \cite{weck} hold, i.e., all spaces in \eqref{eq:hd1}--\eqref{eq:hd3} are closed. Furthermore, the Dirichlet and Neumann fields are absent, i.e., the Helmholtz decompositions \eqref{eq:hd1}--\eqref{eq:hd3} become
\begin{align}
	\lt & = \na\hoc \oplus \dz
			= \rcz \oplus \rot\r ,
		& \na\hoc & = \rcz , & \quad \dz & = \rot\r ,  \label{eq:hdconvex1} \\
	\lt & = \na\ho \oplus \dcz
			= \rz \oplus \, \rot\rc ,
		& \na\ho & = \rz , & \quad \dcz & = \rot\rc .  \label{eq:hdconvex2}
\end{align}
In the following we consider the inequalities
\begin{align}
	& \forall \varphi\in\ho\cap\ltz
	& & \normlt{\varphi} \leq \cp \normlt{\na\varphi} , \nonumber \\
	& \forall \phi \in \rc\cap\dz
	& & \normlt{\phi} \leq \cm \normlt{\rot\phi} , \label{eq:max1} \\
	& \forall \phi \in \r\cap\,\dcz
	& & \normlt{\phi} \leq \cmm \normlt{\rot\phi} , \label{eq:max2}
\end{align}
where the first is the Poincar\'e inequality, and the latter Maxwell-type inequalities. The Poincar\'e constant $\cp>0$ and Maxwell constants $\cm,\cmm>0$ are under the assumptions finite. In what follows, we assume we have chosen the best, i.e., the smallest possible constants in these inequalities. Note that these constants are related to eigenvalues of the Laplace and $\rot\rot$ operators.

The proofs of finiteness of the above constants are based on indirect arguments, and give no hints as to their magnitude. However, in some situations explicit knowledge of these constants is needed: they appear, e.g., in functional type a posteriori error estimates for partial differential equations \cite{repinbookone}. For convex domains there is a constructive method for obtaining an upper bound of $\cp$ due to Payne and Weinberger \cite{payneweinbergerpoincareconvex} (see also \cite{bebendorfpoincareconvex}). The bound is
\begin{equation} \label{eq:cpbound}
	\cp \leq \frac{d}{\pi} ,
\end{equation}
where $d=\diam\Om$ is the diameter of $\Om$. In \cite{paulymaxconst0,paulymaxconst1,paulymaxconst2} Pauly has shown that for convex domains $\cm=\cmm\leq\cp$, so together with \eqref{eq:cpbound} we have
\begin{equation} \label{eq:cm}
	\cm=\cmm\leq\cp \leq \frac{d}{\pi} .
\end{equation}
Using Corollary \ref{cor:D} this upper bound can be slightly improved. However, for the sake of completeness, we first show that the Maxwell constants are indeed equal.

\begin{lem} \label{eq:cmeq}
$\cm=\cmm$.
\end{lem}

\begin{proof}
Let $\phi\in\rc\cap\dz$. From \eqref{eq:hdconvex1}--\eqref{eq:hdconvex2} we deduce $\dz=\rot\r=\rot(\r\cap\,\dcz)$. Thus there exists a vector potential $\Phi \in \r\cap\,\dcz$ such that $\rot\Phi=\phi$. Using \eqref{eq:max2} we obtain
\begin{equation*}
	\normlt{\phi}^2 = \scplt{\phi}{\rot\Phi} = \scplt{\rot\phi}{\Phi}
	\leq \normlt{\rot\phi} \normlt{\Phi}
	\leq \cmm \normlt{\rot\phi} \normlt{\rot\Phi} ,
\end{equation*}
which implies $\normlt{\phi} \leq \cmm \normlt{\rot\phi}$. In view of \eqref{eq:max1} we see that $\cm\leq\cmm$. On the other hand, let $\phi\in\r\cap\,\dcz$. From \eqref{eq:hdconvex1}--\eqref{eq:hdconvex2} we deduce $\dcz=\rot\rc=\rot(\rc\cap\dz)$. Thus there exists a vector potential $\Phi \in \rc\cap\dz$ such that $\rot\Phi=\phi$. Using \eqref{eq:max1} we obtain
\begin{equation*}
	\normlt{\phi}^2 = \scplt{\phi}{\rot\Phi} = \scplt{\rot\phi}{\Phi}
	\leq \normlt{\rot\phi} \normlt{\Phi}
	\leq \cm \normlt{\rot\phi} \normlt{\rot\Phi} ,
\end{equation*}
which implies $\normlt{\phi} \leq \cm \normlt{\rot\phi}$. In view of \eqref{eq:max2} we see that $\cmm\leq\cm$, and the assertion is proven.
\end{proof}

Note that the above proof is not restricted to convex domains. It holds true whenever Weck's selection theorem \cite{weck} holds, provided that the Dirichlet and Neumann fields are excluded from the considered functions.

For improving \eqref{eq:cm} we will need the following specialized Poincar\'e inequality.

\begin{lem} \label{lem:cpscalar}
Let $\Om$ be convex, $\varphi\in\ho(\Om)$ be scalar-valued, and $(i,j,k)\in\{1,2,3\}_3$. Assume $\varphi\in\ltz(\Om_i)$ for an arbitrary $\Om_i$. Then we have
\begin{equation*}
	\norm{\varphi}_{\lt(\Om)} \leq \cpw{i} \norm{\na\varphi}_{\lt(\Om)} ,
	\qquad \cpw{i} \leq \cp ,
	\qquad \cpw{i} \leq \frac{d_{jk}}{\pi} ,
\end{equation*}
where $d_{jk}$ is the diameter of the two-dimensional projection of $\Om$ into the $(e_j,e_k)$-plane. Here $e_j$ and $e_k$ denote the $j$-th and $k$-th Euclidean orthonormal basis vectors (see Figure \ref{fig:diams}).
\end{lem}

\begin{proof}
Let $i=1$. Under the assumptions there exists a decomposition of $\Om$ into nonintersecting convex $\Om_1$-sets $\Om_{1,n}$, $n=1,\ldots,N$ such that
\begin{equation*}
	\ol{\Om} = \bigcup_{n=1}^N \ol{\Om}_{1,n} ,
	\qquad
	\varphi \in \ltz(\Om_{1,n}) ,
	\qquad
	n=1,\ldots,N ,
\end{equation*}
where each $\Om_{1,n}$ has width $l_1/N$ in the direction of the $x_1$-coordinate, and
\begin{equation*}
	\forall n\in\{1,\ldots,N\} \qquad
	\diam\Om_{1,n}
	\leq \sqrt{d_{23}^2 + \frac{l_1^2}{N^2}}
\end{equation*}
holds (see Figure \ref{fig:diams}).
For each subdomain we can apply \eqref{eq:cpbound} to obtain
\begin{equation*}
	\forall n\in\{1,\ldots,N\} \qquad
	\norm{\varphi}_{\lt(\Om_{1,n})}
	\leq \frac{\diam\Om_{1,n}}{\pi} \norm{\na\varphi}_{\lt(\Om_{1,n})} ,
\end{equation*}
which implies
\begin{equation*}
	\norm{\varphi}_{\lt(\Om)}
	\leq \frac{1}{\pi} \max_{n\in\{1,\ldots,N\}}\diam\Om_{1,n} \norm{\na\varphi}_{\lt(\Om)}
	\leq \frac{1}{\pi} \sqrt{d_{23}^2 + \frac{l_1^2}{N^2}} \norm{\na\varphi}_{\lt(\Om)}
	\xrightarrow{N\to\infty} \frac{d_{23}}{\pi} \norm{\na\varphi}_{\lt(\Om)} .
\end{equation*}
The cases $i=2$ and $i=3$ are proven in a similar way.
\end{proof}

\begin{figure}
\centering
\tdplotsetmaincoords{55}{105}
%####################################################################
% AXIS
%####################################################################
\begin{tikzpicture}[scale=0.2,tdplot_main_coords]
	% axis (from the origin to positive direction)
	%------
	\draw [black,line width=1pt,->] (0,0,0) -- (-4,0,0) node[anchor=west]{$x_2$};
	\draw [black,line width=1pt,->] (0,0,0) -- (0,3,0) node[anchor=north]{$x_1$};
	\draw [black,line width=1pt,->] (0,0,0) -- (0,0,3.5) node[anchor=east]{$x_3$};
\end{tikzpicture}
%####################################################################
% diameters
%####################################################################
\begin{tikzpicture}[scale=1.1,tdplot_main_coords]
	\def\len{4}

	% dotted lines first
	\draw [black,line width=0.7pt,dotted] (0,0,0) -- (0,0,1); % vertical
	\draw [black,line width=0.7pt,dotted] (0,\len,0) -- (0,0,0) -- (1,0,0); % lower level

	% diameters
	\draw [black,line width=0.5pt] (1,0,1) -- (0,\len,1);			% d_12
	\draw [black,line width=0.5pt] (1,0,1) -- (1,\len,0);			% d_23
	\draw [black,line width=0.5pt] (1,\len,0) -- (0,\len,1);		% d_13

	% upper level
	\draw [black,line width=1pt] (0,0,1) -- (1,0,1) -- (1,\len,1) -- (0,\len,1) -- cycle;
	
	% lower level
	\draw [black,line width=1pt] (1,0,0) -- (1,\len,0) -- (0,\len,0);
	
	% vertical lines
	\draw [black,line width=1pt] (1,0,0) -- (1,0,1);
	\draw [black,line width=1pt] (1,\len,0) -- (1,\len,1);
	\draw [black,line width=1pt] (0,\len,0) -- (0,\len,1);

	% l_1
	\draw [black,line width=0.7pt] (1.1,0,0) -- (1.25,0,0) -- (1.25,\len,0) -- (1.1,\len,0);
	
	% text
	\draw [black,line width=0.7pt] (-0.1,\len/5-0.1,1) node[anchor=north]{$d_{12}$};
	\draw [black,line width=0.7pt] (0.1,\len/3-0.1,0) node[anchor=north]{$d_{13}$};
	\draw [black,line width=0.7pt] (0,\len+0.35,0.5) node[anchor=north]{$d_{23}$};
	\draw [black,line width=0.7pt] (1.25,\len/2,0) node[anchor=north]{$l_1$};
\end{tikzpicture}
%####################################################################
% \Om_2 decomposition
%####################################################################
\begin{tikzpicture}[scale=1.1,tdplot_main_coords]
	\def\len{4}

	% dotted lines first
	\draw [black,line width=0.7pt,dotted] (0,0,0) -- (0,0,1); % vertical
	\draw [black,line width=0.7pt,dotted] (0,\len,0) -- (0,0,0) -- (1,0,0); % lower level
	
	% diam(\Om_{1,2})
	\draw [black,line width=0.7pt] (1,\len/4,0) -- (0,\len/2,1);
	
	% Decomposition into \Om_2-sets: dotted lines
	\draw [black,line width=0.7pt,dotted] (1,\len/4,0) -- (0,\len/4,0) -- (0,\len/4,1);
	\draw [black,line width=0.7pt,dotted] (1,2*\len/4,0) -- (0,2*\len/4,0) -- (0,2*\len/4,1);
	\draw [black,line width=0.7pt,dotted] (1,3*\len/4,0) -- (0,3*\len/4,0) -- (0,3*\len/4,1);
		
	% Decomposition into \Om_2-sets: solid lines
	\draw [black,line width=0.7pt] (1,\len/4,0) -- (1,\len/4,1) -- (0,\len/4,1);
	\draw [black,line width=0.7pt] (1,2*\len/4,0) -- (1,2*\len/4,1) -- (0,2*\len/4,1);
	\draw [black,line width=0.7pt] (1,3*\len/4,0) -- (1,3*\len/4,1) -- (0,3*\len/4,1);

	% upper level
	\draw [black,line width=1pt] (0,0,1) -- (1,0,1) -- (1,\len,1) -- (0,\len,1) -- cycle;
	
	% lower level
	\draw [black,line width=1pt] (1,0,0) -- (1,\len,0) -- (0,\len,0);
	
	% vertical lines
	\draw [black,line width=1pt] (1,0,0) -- (1,0,1);
	\draw [black,line width=1pt] (1,\len,0) -- (1,\len,1);
	\draw [black,line width=1pt] (0,\len,0) -- (0,\len,1);
	
	% l_1\4
	\draw [black,line width=0.7pt] (1.1,\len/4,0) -- (1.25,\len/4,0) -- (1.25,\len/2,0) -- (1.1,\len/2,0);
	
	% text
	\draw [black,line width=0.7pt] (0.1,\len/4-0.55,1) node[anchor=north]{$\Om_{1,1}$};
	\draw [black,line width=0.7pt] (0.1,3*\len/4-0.55,1) node[anchor=north]{$\Om_{1,3}$};
	\draw [black,line width=0.7pt] (0.1,4*\len/4-0.55,1) node[anchor=north]{$\Om_{1,4}$};
	\draw [black,line width=0.7pt] (1.25,\len/3,0) node[anchor=north]{$l_1/4$};
	\draw [black,line width=0.7pt] (0.2,1.7,0.8) .. controls (-0.5,1.4,0.8) .. (-0.3,2,1) node[anchor=west]{$\diam\Om_{1,2}$};
\end{tikzpicture}
\caption{Diameters of a rectangular cuboid $\Om$ and its decomposition into $\Om_1$-sets.}
\label{fig:diams}
\end{figure}
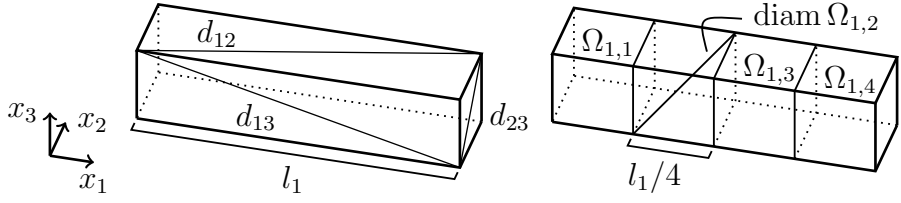

As in \cite{paulymaxconst0,paulymaxconst1,paulymaxconst2}, we will rely on the essential regularity result \cite[Thm. 2.17]{amrouche98}.

\begin{lem} \label{lem:horeg}
Let $\Om$ be convex and
$\phi\in\rc\cap\d$ or
$\phi\in\r\cap\,\dc$. Then $\phi\in\ho$ and
\begin{equation*}
	\normlt{\na\phi}^2 \leq \normlt{\div\phi}^2 + \normlt{\rot\phi}^2 .
\end{equation*}
\end{lem}

We can now state the improved bound.

\begin{thm} \label{thm:cm}
Let $\Om$ be convex. Then we have the estimate
\begin{equation*}
	\cm=\cmm \leq \max\{\cpw{1},\cpw{2},\cpw{3}\}
	\leq \frac{\max\{d_{23},d_{13},d_{12}\}}{\pi} .
\end{equation*}
\end{thm}

\begin{proof}
Let $\phi\in\r\cap\,\dcz$. Then $\phi \in \ho$ by Lemma \ref{lem:horeg} and $\phi \in \ltz$ by Corollary \ref{cor:D}. More specifically, Corollary \ref{cor:D} shows that the specialized Poincar\'e inequality of Lemma \ref{lem:cpscalar} can be applied to each component of $\phi$, and we directly get
\begin{align*}
	\normlt{\phi}^2
	& = \normlt{\phi_1}^2 + \normlt{\phi_2}^2 + \normlt{\phi_3}^2
		\leq \cpw{1}^2\normlt{\na\phi_1}^2 + \cpw{2}^2\normlt{\na\phi_2}^2 + \cpw{3}^2\normlt{\na\phi_3}^2 \\
	& \leq \max\{\cpw{1}^2,\cpw{2}^2,\cpw{3}^2\} \normlt{\na\phi}^2
		\leq \max\{\cpw{1}^2,\cpw{2}^2,\cpw{3}^2\} \normlt{\rot\phi}^2 ,
\end{align*}
where in the last step we used Lemma \ref{lem:horeg}. In view of \eqref{eq:max2} we obtain $\cmm\leq\max\{\cpw{1},\cpw{2},\cpw{3}\}$. Together with Lemmas \ref{eq:cmeq} and \ref{lem:cpscalar} we have the assertion.
\end{proof}

\begin{rem}
If we had used in the above proof the global zero mean property \eqref{eq:Dglobal} and the Payne-Weinberger estimate \eqref{eq:cpbound} (instead of Corollary \ref{cor:D} and Lemma \ref{lem:cpscalar}, respectively), we would have arrived at \eqref{eq:cm}. Note that Pauly's proof of \eqref{eq:cm} does not use knowledge of \eqref{eq:Dglobal}, but is rather based on finding suitable potential functions.
\end{rem}

\begin{exam}
\mbox{}
\begin{itemize}
\item[\bf(i)] Let $\Om=(0,1)^3$. Then $d=\sqrt{3}$ and $d_{23}=d_{13}=d_{12}=\sqrt{2}$. The bounds of \eqref{eq:cm} and Theorem \ref{thm:cm} then give
\begin{equation*}
	\cm=\cmm\leq\frac{\sqrt{3}}{\pi} , \qquad \cm=\cmm\leq\frac{\sqrt{2}}{\pi} ,
\end{equation*}
respectively.
\item[\bf(ii)] Let $\Om=B(0,1)$, i.e., the unit ball in $\reals^3$. Then $d=d_{23}=d_{13}=d_{12}=2$, and the bound in Theorem \ref{thm:cm} offers no improvement over the bound \eqref{eq:cm}.
\end{itemize}
\end{exam}

\begin{rem}
In \cite{paulymaxconst1} it was proven that for convex domains $\Om$ the two Maxwell constants in the inequalities
\begin{align*}
	& \forall \phi\in\rc\cap\d
	& & \normlt{\phi}^2 \leq \cmt^2 \brac{\normlt{\div\phi}^2+\normlt{\rot\phi}^2} , \\
	& \forall \phi\in\r\cap\,\dc
	& & \normlt{\phi}^2 \leq \cmn^2 \brac{\normlt{\div\phi}^2+\normlt{\rot\phi}^2} ,
\end{align*}
satisfy $\cmt \leq \cmn = \cp$, and it was conjectured that $\cmt<\cmn$ holds. By using Theorem \ref{thm:cm} instead of \cite[Lemma 4]{paulymaxconst1} in the proof of  \cite[Theorem 6]{paulymaxconst1}, we obtain
\begin{align*}
	& \forall \phi\in\rc\cap\d
	& & \normlt{\phi}^2 \leq \cf^2 \normlt{\div\phi}^2 + \max\{\cpw{1},\cpw{2},\cpw{3}\}^2 \normlt{\rot\phi}^2 , \\
	& \forall \phi\in\r\cap\,\dc
	& & \normlt{\phi}^2 \leq \cp^2 \brac{\normlt{\div\phi}^2+\normlt{\rot\phi}^2} ,
\end{align*}
where $\cf$ is the constant in the Friedrichs' inequality $\normlt{\varphi} \leq \cf \normlt{\na\varphi}$ which holds for all scalar valued functions $\varphi\in\hoc$. It is well known that $\cf<\cp$ (see, e.g., \cite{filonov05}). Thus, if one can prove that $\max\{\cpw{1},\cpw{2},\cpw{3}\}<\cp$, then the conjecture $\cmt<\cmn$ follows.

Note also that weighted $\lt$-orthogonal Helmholtz decompositions were used in \cite{paulymaxconst1}. In this note unweighted decompositions were used only for simplicity.
\end{rem}

%====================================================================
%====================================================================

\bibliographystyle{plain} 
\bibliography{biblio}

\end{document}